\documentclass[a4paper]{article}
\usepackage[all]{xy}\usepackage[latin1]{inputenc}        
\usepackage[dvips]{graphics,graphicx}
\usepackage{amsfonts,amssymb,amsmath,color,mathrsfs, amstext}
\usepackage{amsbsy, amsopn, amscd, amsxtra, amsthm,authblk}
\usepackage{enumerate,algorithmic,algorithm}
\usepackage{upref}
\usepackage{geometry}
\usepackage[displaymath]{lineno}
\usepackage{float}
\usepackage{yhmath}
\usepackage{booktabs}
\usepackage{subcaption}
\usepackage{multirow}
\usepackage{makecell}
\usepackage[colorlinks,
            linkcolor=red,
            anchorcolor=red,
            citecolor=red
            ]{hyperref}

\numberwithin{equation}{section}

\def\e{\epsilon}

\def\R{\mathbb{R}}

\def\cA{\mathcal{A}}

\def\pc{\bar{*}}

\newtheorem{theorem}{Theorem}[section]
\newtheorem{lemma}{Lemma}[section]
\newtheorem{proposition}{Proposition}[section]
\newtheorem{remark}{Remark}[section]
\newtheorem{corollary}{Corollary}[section]
\newtheorem{definition}{Definition}[section]

\makeatletter
\@mparswitchfalse
\makeatother

\title{A class of monotonicity-preserving variable-step discretizations for Volterra integral equations and time fractional ordinary differential equations}

\author[a]{Yuanyuan Feng\thanks{
E-mail: yyfeng@math.ecnu.edu.cn}}
\author[b]{Lei Li\thanks{E-mail: leili2010@sjtu.edu.cn}}

\affil[a]{School of  Mathematical Sciences, Shanghai Key Laboratory of PMMP, East China Normal University, Shanghai, 200241, P.R. China. }
\affil[b]{School of Mathematical Sciences, Institute of Natural Sciences, MOE-LSC, Shanghai Jiao Tong University, Shanghai, 200240, P.R.China.}

\date{}

\begin{document}

\maketitle

\begin{abstract}
We study in this paper the monotonicity properties of the numerical solutions to Volterra integral equations with nonincreasing completely positive kernels on nonuniform meshes.  There is a duality between the complete positivity and the properties of the complementary kernel being nonnegative and nonincreasing. Based on this, we propose the ``complementary monotonicity''  to describe the nonincreasing completely positive kernels, and the ``right complementary monotone'' (R-CMM) kernels as the analogue for nonuniform meshes. We then establish the monotonicity properties of the numerical solutions inherited from the continuous equation if the discretization has the R-CMM property.  Such a property seems weaker than being log-convex and there is no resctriction on the step size ratio of the discretization for the R-CMM property to hold.
\end{abstract}


\section{Introduction}

The time-delay memory is ubiquitous in physical models, which may be resulted from dimension reduction as in the generalized Langevin model for particles in heat bath (\cite{zwanzig73,zwanzig01,kouxie04,li2017fractional}) or may be resulted from viscoelasticity in soft matter (\cite{colemannoll1961,pd97}), or dielectric susceptibility for polarization \cite{stenzel05,cai13}, to name a few examples. Due to the causality and time translational invariance \cite[Chap. 1]{nussenzveig72}, the memory terms are often modeled by a one-sided convolution $\int_0^t a(t-s)f(s)\,ds$ where $a$ is the memory kernel.  Causality refers to the fact that the output cannot precede the input, which implies, by  Tichmarsh's theorem, that the Fourier transform of $a$ is analytic in the upper half plane so that the real and imaginary parts satisfy the Kramers-Kronig relation (see \cite{nussenzveig72}). Besides the causality, the memory kernel $a$ should also reflect the fading memory principle \cite{wang1965principle,pd97}. A popular model to build in these physical principles would be the nonincreasing, completely positive kernels, which are a class of kernels with nonnegative resolvent kernels (see \cite{clement1981asymptotic,miller1968volterra} and see also section \ref{subsec:timecontinuous} for the definitions). These kernels have been proved to reflect many important asymptotic properties that the system is expected to have \cite{clement1981asymptotic}. 
A special but important class of the nonincreasing completely positive kernels is the completely monotone (CM) functions, which have been well studied in literature  \cite{widder41,ssv12}.  The CM functions have been widely used in physical modeling. For example, the interconversion relationship in the linear viscoelasticity is modeled by a convolution quadrature with completely monotone kernels \cite{loy2014interconversion}.  There are many interesting models with memory in literature for various applications \cite{cuesta2006convolution,tang2019energy,quan2020define,zhan2019complete,bonaccorsi2012optimal}.

A basic model for the memory is the Volterra integral equations (see \cite{gripenberg1990volterra,miller1971smoothness,weis1975asymptotic,loy2014interconversion}).  In this work, we focus on the Volterra integral equations taking values in $\R$. Let $f: [0,\infty)\times \R\to \R$ be a given smooth function. The integral equation we consider in this work is
\begin{gather}\label{eq:vol}
u(t)=h(t)+\int_0^t a(t-s) f(s, u(s))\,ds,
\end{gather}
where $u: [0, T)\to \R$ is the solution curve. The function $h(t)$ is a given signal function.  The function $a(\cdot)$ is the convolution kernel and is assumed to be nonzero. We will allow $a$ to be weakly singular in the sense that $a(0+)$ could be $\infty$ but it is integrable on $(0, 1)$:
\begin{gather}
0<\int_0^1 a(t)dt<\infty.
\end{gather}
A special example of the integral equation \eqref{eq:vol} is the time fractional ordinary differential equations (FODEs) with Caputo derivative  \cite{diethelm10} of order $\alpha\in (0, 1)$
\begin{gather}\label{eq:fracode}
D_c^{\alpha}u=f(t, u), \quad u(0)=u_0.
\end{gather}
Here, the Caputo derivative is defined by
\begin{gather}\label{eq:captra}
D_c^{\alpha}u=\frac{1}{\Gamma(1-\alpha)}\int_0^t\frac{u'(s)}{(t-s)^{\alpha}}ds.
\end{gather}
See also \cite{liliu18frac,liliu2018compact} for some generalized definitions. The time fractional ODE \eqref{eq:fracode} is equivalent to the integral equation (see \cite{diethelm10,liliu2018compact} etc)
\begin{gather}\label{eq:fracint}
u(t)=u_0+\frac{1}{\Gamma(\alpha)}\int_0^{t}(t-s)^{\alpha-1}f(s,u(s))ds.
\end{gather}
Hence, the FODEs are Volterra equations with kernel
\begin{gather}\label{eq:kernelfode}
a(t)=\frac{1}{\Gamma(\alpha)}t^{\alpha-1}, \quad t>0.
\end{gather}
If $a\equiv 1$, it reduces to the usual ODE, which is in fact the $\alpha\to 1$ limit of the above time fractional equation.

We will focus on kernels $a$ that are nonincreasing and  completely positive \cite{clement1981asymptotic}.  The Volterra equation \eqref{eq:vol} with complete positive kernels (not necessarily nonincreasing) have two important monotonicity properties. The first is that the order of two solution curves will stay the same (i.e., two solution curves will not cross) for reasonable given input signals $h$. A second monotonicity property is that the solution to the autonomous equations is monotone for reasonable given signals (see section \ref{subsec:contcmm} below for the concrete requirement on the input signals).  The nonincreasing property cannot be implied by the complete positivity, but is required by the fading memory principle.

Due to the memory kernels, especially some weakly singular kernels, the models often exhibit multi-scale behaviors \cite{cuesta2006convolution,tang2019energy,zhan2019complete}, which bring numerical challenges. The adaptive time-stepping is often adopted to address this issue \cite{mclean1996discretization,kopteva2019error,liao2019discrete,stynes2017error,li2019linearized}.  Suppose that the computational time interval is $[0, T]$. Let $0=t_0<t_1<t_2<\cdots<t_N=T$ be the grid points.  We define
\begin{gather}
\tau_n:=t_n-t_{n-1}, \quad n\ge 1.
\end{gather}
Let $u_n$ be the numerical solution at $t_n$.
By implicit discretization of the Volterra integral equation \eqref{eq:vol}, one may obtain
\begin{gather}\label{eq:integraldis}
u_n=h(t_n)+\sum_{j=1}^n \bar{a}_{n-j}^n f(t_j, u_j)\tau_j=h(t_n)+\sum_{j=1}^n a_{n-j}^n f(t_j, u_j).
\end{gather}
Here, $\{\bar{a}_{n-j}^n\}$ is an approximation of $a(t_n-s)$ on $[t_{j-1}, t_j]$ while
$a_{n-j}^n$ is like the inegral of $a(t_n-s)$ on this interval. 
It is clear that
\begin{lemma}
Suppose $f(t, \cdot)$ is Lipschitz with Lipschitz constant $M$ uniform in $t$. If $M\sup_{j\le N}\tau_j \bar{a}_0^j<1$, then the numerical solution to \eqref{eq:integraldis} is uniquely solvable.
\end{lemma}
From the viewpoint of structure-preserving methods, it is desired that the discrete numerical methods preserve the monotonicity properties of the solutions.
Our main goal is to investigate this for the scheme above.

In \cite{li2021complete},  the so-called CM-preserving schemes have been proposed for equations with CM kernels so that these two monotonicity properties can be preserved. These schemes have been shown to enjoy good stability properties. Recently, Chen and Stynes used the CM-preserving property to obtain sharp error estimate for the multi-term time-fractional diffusion equations \cite{chen2022using}.  However, the CM-preserving schemes are based on the convolution and thus they are restricted on uniform meshes. Moreover, if the time continuous kernel itself is not CM, there is no reason to consider such a class of discretizations.

The goal of this paper is to identify a suitable class of variable discretizations that are discrete analogue  for the nonincreasing completely positive kernels, so that the monotonicity properties for numerical solutions will be preserved. Motivated by a characterization of completely positive kernels due to Clement and Nohel  \cite{clement1981asymptotic}, we propose ``right complementary monotone'' (R-CMM) kernels for nonuniform meshes inherited from the continuous kernel. This turns out to be a suitable requirement for our purpose. In particular, we prove that if the discrete kernel satisfies this property, one solution is always bigger if the initial value is bigger, and the numerical solutions to the autonomous equations are monotone. It is found that there is no restriction on the ratios of the stepsizes for the R-CMM property to hold and it is weaker than the usual condition that $a$ is log-convex. 

The rest of the paper is organized as follows. In section \ref{sec:continuous}, we first present some concepts and results for the time continuous kernels and equations. Then, motivated by the duality between the complete positivity and the properties of the complementary kernel, we propose the complementary monotoncity (CMM).  In section \ref{sec:cmmunifom}, we consider the CMM property on uniform meshes. In section \ref{sec:Rcmmnonuniform}, we introduce the fundamental tools we use in this work, namely the R-CMM property for kernels on nonuniform meshes. In section \ref{sec:monpre}, we prove the main results in this paper, namely the schemes whose kernels are R-CMM will preserve the two monotonicity properties. A simple but important example for the FODEs is provided for illustration in section \ref{sec:ex}.

\section{Time continuous kernels and equations}\label{sec:continuous}

In this section, we consider the time continuous kernels and the time continuous integral equations. In section \ref{subsec:timecontinuous}, we review the preliminary concepts and show the monotonicity results of the Volterra integral equations with completely positive kernels. In section \ref{subsec:contcmm}, we discuss the duality between complete positivity and the nonnegativity and nonincreasing property of the complementary kernels, and then propose the concept of complementary monotonicity.

\subsection{The resolvent kernels and complete positivity}\label{subsec:timecontinuous}

We first recall the standard one-sided convolution for two functions $u$ and $v$ defined on $[0,\infty)$
\begin{gather}
u*v(t)=\int_{[0, t]}u(s)v(t-s)\,ds,
\end{gather}
which can be generalized to distributions whose supports are on $[0,\infty)$ (see \cite[sections 2.1,2.2]{liliu18frac}). The convolution is commutative, associative. The identity is the Dirac delta $\delta$.
With this convolution, the Volterra integral equation \eqref{eq:vol} can be written as
\[
u=h+a*f(\cdot, u(\cdot)).
\]
The resolvent kernels often play important roles (see \cite{clement1981asymptotic,miller1968volterra} for examples).
\begin{definition}\label{def:resol}
Let $\lambda>0$. The resolvent kernels $r_{\lambda}$ and $s_{\lambda}$ for $a$ are defined respectively by
\begin{gather}
r_{\lambda}+\lambda r_{\lambda}*a=\lambda a, \quad 
 s_{\lambda}+\lambda s_{\lambda}* a=1.
\end{gather}
\end{definition}
Clearly, the resolvent kernel $r_{\lambda}$ satisfies
\begin{gather}\label{eq:interres}
(\delta +\lambda a)*(\delta-r_{\lambda})=\delta.
\end{gather}
It is clearly that (see \cite{clement1981asymptotic})
\begin{gather}
s_{\lambda}=1*(\delta -r_{\lambda})=1-\int_0^t r_{\lambda}(\tau)\,d\tau.
\end{gather}
Intuitively,  $\delta-r_{\lambda}=\lambda^{-1}r_{\lambda}*a^{(-1)}$. Note that the convolutional inverse $a^{(-1)}$ is not clear at this point, but this gives $s_{\lambda}=\lambda^{-1}r_{\lambda}*a^c$,
where $a^c$ is the complementary kernel appeared below in Lemma \ref{lmm:contcp} and Definition \ref{def:contcmm} later.

In \cite{clement1981asymptotic},  the so-called ``completely positive'' kernels were considered by Clement and Nohel.
\begin{definition}\label{def:cpcont}
Let $T>0$. A kernel $a\in L^1(0, T)$ is said to be completely positive if both the resolvent kernels $r_{\lambda}$ and $s_{\lambda}$ defined in Definition \ref{def:resol} are nonnegative for every $\lambda>0$.
\end{definition}

A sufficient condition is the following (see \cite{miller1968volterra}).
\begin{lemma}\label{lmm:continouslogconv}
If the kernel $a\in L^1(0, T)$ is nonnegative, nonincreasing and $t\mapsto \log a(t)$ is convex, then $a$ is completely positive.
\end{lemma}
In fact, the statement for the log-convexity of $a$ in  \cite{miller1968volterra} is that $t\mapsto a(t)/a(t+T)$ is nonincreasing for all $T>0$. If $a$ is CM, $\log a$ is convex (see \cite[Lemma 2]{miller1968volterra}). 

The following description of the complete positivity has been proved in \cite[Theorem 2.2]{clement1981asymptotic}. (The second claim has been mentioned in Remark (i) below the main result there.)  
\begin{lemma}\label{lmm:contcp}
Let $T>0$. A kernel $a\in L^1(0, T)$ with $a\not\equiv 0$ is completely positive on $[0, T]$ if and only if there exists $\alpha\ge 0$ and $c \in L^1(0, T)$ nonnegative and nonincreasing  satisfying
\begin{gather}\label{eq:cpchar}
\alpha a+c*a=a*(\alpha \delta+c)=1_{t\ge 0}.
\end{gather}
Moreover, provided that $a$ is completely positive, $\alpha>0$ if and only if $a\in L^{\infty}(0, T)$ and in this case $a$ is in fact absolutely continuous on $[0, T]$.
\end{lemma}

This result tells us that there is a complementary kernel $a^c=\alpha \delta+c$ for $a$. Cearly, $a^c$ is a  nonnegative and nonincreasing measure on $[0,T]$.  We note that a completely positive kernel $a$ is nonnegative (see \cite[Proposition 2.1]{clement1981asymptotic}). However, it has been remarked in \cite{clement1981asymptotic} that $a$ can increase at some subintervals, and also $a$ does not have to be convex. Hence, the complete positivity somehow violates the requirement of fading memory.

Now, we present some monotonicity properties of the time continuous equations with completely positive kernels. We will always assume that $h\in C([0, T], \R)\cap C^1((0, T], \R)$ where $T>0$ is the largest time considered. If $a\in L^1(0, T)$ and $f$ is smooth as assumed, then $u$ is absolutely continuous on $[0, T_b)$ where $T_b\le T$ is the largest time of existence (see \cite{gripenberg1990volterra,weis1975asymptotic}).  

The first result is about two solution curves.
\begin{theorem}\label{thm:contcomparison}
Suppose the kernel $a$ is completely positive. If the input signals $h_i\in C([0, T], \R)\cap C^1((0, T], \R)$ ($i=1,2$) and $\gamma(t):=h_1(t)-h_2(t)$ satisfies that
\[
\beta_{\lambda}(t):=(\delta-r_{\lambda})*\gamma(t)=\gamma(t)-\int_0^t r_{\lambda}(t-s)\gamma(s)\,ds \ge 0,\quad \forall \lambda>0,
\]
then the two solutions to \eqref{eq:vol} satisfy that $u_1(t)\ge u_2(t)$ for all $t$ on the common interval of existence. If moreover $\gamma(0)>0$, then $u_1(t)>u_2(t)$ for all $t$. Consequently, if $h_1$ and $h_2$ are two constants with $h_1>h_2$, then $u_1(t)>u_2(t)$.
\end{theorem}

The key in the proof is to convolving the integral equation with $\delta-r_{\lambda}$ to obtain the following relation for $v=u_1-u_2$:
\begin{gather}
v(t)=\beta_{\lambda}(t)+\int_0^t r_{\lambda}(t-s)[1+\lambda^{-1}g(s)]v(s)\,ds.
\end{gather}
This will give the nonnegativity of $v$ and see the details in Appendix \ref{app:monoproof}.  For fractional ODEs, $h_i(t)$'s are constants so $\beta_{\lambda}\ge 0$ is obvious.  Hence, the solution curves of the fractional ODEs never cross each other. This recovers the result in  \cite[Theorem 4.1]{fllx18note}.

Next, we consider another monotonicity, the monotonicity of one solution with respect to time.
\begin{theorem}\label{thm:mon4auto}
Consider the Volterra equation \eqref{eq:vol} and a solution $u$ corresponding to an input signal $h\in C([0, T], \R)\cap C^1((0, T], \R)$. Suppose that $a$ is completely positive. 
\begin{enumerate}[(i)]
\item If $(\delta-r_{\lambda})* h'\ge 0$, $f(0, h(0))\ge 0$ and $\partial_t f(t, u)|_{u=u(t)}\ge 0$ on the solution curve, then the solution is nondecreasing.
\item If $(\delta-r_{\lambda})* h'\le 0$, $f(0, h(0))\le 0$ and $\partial_t f(t, u)|_{u=u(t)}\le 0$ on the solution curve, then the solution is nonincreasing.  
\end{enumerate}
If any of the inequalities is strict, then $u$ is strictly monotone.
As a consequence, if $h$ is a constant and the equation is autonomous ($f$ only depends on $u$), then any solution curve is monotone. 
\end{theorem}

For the proof, one may refer to Appendix \ref{app:monoproof}. The basic idea is again to convole the equation with $\delta-r_{\lambda}$ and take the derivative on time to obtain an equation for $u'(t)$:
\[
u'(t)=g(t)+r_{\lambda}*([1+\lambda^{-1}\partial_u f] u'),
\]
Here, $g$ is a term related to $h$ and the initial value, $\partial_t f(t, u)$. If $h$ is a constant and the equation is autonomous, $g(t)=\lambda^{-1}r_{\lambda}(t)f(h(0))$.  This clearly is a generalization of the result for autonomous time fractional ODEs in \cite[Theorem 3.3]{feng2018continuous}. For the special case $a(t)\equiv 1$, $h(t)\equiv u_0$ and $f(t, u)=f(u)$, it reduces to the autonomous ODE
\[
\dot{u}=f(u).
\]
Then, $r_{\lambda}=\lambda e^{-\lambda t}$. Equation \eqref{eq:uprime} reduces to
\[
u'(t)=e^{-\lambda t}f(u_0)+\int_0^t e^{-\lambda (t-s)}(\lambda+ f'(u(s)))u'(s)\,ds.
\]
It is clear that the right hand is in fact equal to $f(u(t))$.
Such a form is interesting as one can see easily that $u'(t)$ has the same sign as $f(u_0)$, which implies that $u$ is monotone.

\subsection{Time continuous complementary monotone kernels}\label{subsec:contcmm}

Consider the discretization \eqref{eq:integraldis}. We find that when it is uniform, i.e. $\tau_n\equiv \tau$ is a constant and  $a_{n-j}^n \equiv a_{n-j}$ depends only on the value $n-j$, the complete positivity (see \cite{fengli2023a} or Definition \ref{def:uniformcp} below), the monotonicity properties can be preserved for the numerical solutions (see Remark \ref{rmk:cmtuniform} and \cite{li2021complete} for similar proof).
However, for the nonuniform meshes, the complete positivity introduce in \cite{fengli2023a} is not enough. Fortunately, as we shall see in section \ref{subsec:monpreserving}, if we borrow a little bit more property by requiring the kernel to be nonincreasing, then the monotonicity properties can be established for the nonuniform case as well, which is natural as the kernel is often nonincreasing in physical models due to the fading memory.

Lemma \ref{lmm:contcp} gives a nice characterization of the duality between the complete positivity and the nonnegativity and nonincreasing property. If we the kernel itself is nonincreasing (and it is already known to be nonnegative), then $a^c$ is completely positive. This observation gives a nice symmetric property of the kernels, which we call ``complementary monotonicity'' in the sense that both the kernel and the complementary kernel are nonnegative and monotone, and also completely positive. 

Motivated by Lemma \ref{lmm:contcp}, we will consider kernels of the following form:
\begin{gather}
\cA:=\{a=\alpha \delta+\tilde{a}: \alpha\ge 0,\quad \tilde{a}\text{ is integrable on}~[0, T]  \}.
\end{gather}
For such kernels,  $a\in \cA$ is nonincreasing if and only if $\tilde{a}$ is nonincreasing. It is nonnegative if and only if $\tilde{a}$ is nonnegative. We then define the following.
\begin{definition}\label{def:contcmm}
A kernel $a\in \cA$ with $a\not\equiv 0$ is said to be complementary monotone (CMM) if $a$ is nonnegative and nonincreasing and there exists a kernel $a^c\in \cA$ that is nonnegative and nonincreasing such that $a*a^c=1_{t\ge 0}$.
\end{definition}

Using Lemma \ref{lmm:contcp}, one can show an analogue of Theorem \ref{thm:uniformcmm} as follows, where a completely positive kernel is allowed to be in $\cA$ (we allow an atom at $t=0$ and no longer require it to be an $L^1$ function).
\begin{proposition}\label{pro:cmmcontdes}
Fix $T>0$. The following are equivalent.
\begin{enumerate}[(a)]
\item A kernel $a\in \cA$ is CMM on $[0, T]$.
\item  The kernel $a$ is nonincreasing and is completely positive on $[0, T]$.
\item The complementary kernel $a^c\in \cA$ exists, and is nonincreasing, completely positive on $[0, T]$.
\end{enumerate}
\end{proposition}

Note that in \cite{clement1981asymptotic}, the kernel does not have an atom at $t=0$. Here, we allow an atom at $t=0$. However, the proof for the properties of the complementary kernels and resolvents in \cite[Theorem 2.2]{clement1981asymptotic} is actually valid as well. We sketch the proof in Appendix \ref{app:contcmm} for the convenience of the readers. 

The result above indicates that the nonincreasing property of the kernel is kind of crucial for the complete positivity of the complementary kernel. This indicates that the complete positivity of the complementary kernel may have given a description to the fading memory principle. 

\begin{remark}
Since the complementary kernel is $a^{(-1)}*1_{t\ge 0}$. Then, the above result actually indicates that $a^{(-1)}$ exists for kernels in $\cA$, given by 
\[
a^{(-1)}=\alpha^c \delta'+(\tilde{a}^c)',
\] 
where $a^c=\alpha^c\delta+\tilde{a}^c$ and the derivative is in distributional sense. If $\tilde{a}^c(0+)$ exists and is nonzero, then there is an additional $\delta$ in the convolutional inverse.
\end{remark}

\section{Complementary monotone kernels on uniform meshes}\label{sec:cmmunifom}

In this section, we focus on the discrete analogue of complementary monotonicity for the uniform meshes, i.e., $\tau_n\equiv \tau$ and thus $t_j=j\tau$.  
This turns out to be the nonincreasing property plus the complete positivity. As remarked, though the complete positivity is already enough for the numerical solutions to preserve the desired properties, the nonincreasing property is required in applications by fading  memory principle, and moreover complementary montonicity itself is interesting enough even for uniform meshes.  Besides, the discussion here is necessary for the study later for nonuniform meshes.

 For sequences on uniform mesh, we recall the usual convolution, which is commutative,
\begin{gather}
(a*b)_n=\sum_{j=0}^n a_{n-j}b_j.
\end{gather}
It is clear that $\delta_d=(1, 0, 0, \cdots)$ is the convolution identity and the convolution inverse of $a$ exists if and only if $a_0\neq 0$.

Corresponding to Definition \ref{def:cpcont}, the complete positivity for uniform meshes was introduced in \cite{fengli2023a} as follows.
\begin{definition}\label{def:uniformcp}
A sequence $a=(a_0, a_1, \cdots)$ with $a_0\neq 0$ is said to be completely positive if the resolvent sequence given by 
\[
r_{\lambda}+\lambda r_{\lambda}*a=\lambda a
\]
is nonnegative for all $\lambda>0$ and it holds that $\sum_{i=0}^n (r_{\lambda})_i\le 1$ for all $n$.
\end{definition}

It has been proved in \cite{fengli2023a} that 
\begin{lemma}\label{lmm:discretecpuni}
The sequence $a$ with $a_0\neq 0$ is completely positive if and only if 
the convolutional inverse $b=a^{(-1)}$ satisfies
\begin{gather}\label{eq:signproperty}
b_0>0; \quad b_j\le 0,  j\ge 1;  \quad \sum_{j=0}^n b_j \ge 0, n\ge 1.
\end{gather}
\end{lemma}
This result acutally is an analogue to Lemma \ref{lmm:contcp}. This is because $a^c=b*(1,1,\cdots)$. Hence, the nonnegativity and nonincresing property of $a^c$ is reflected by \eqref{eq:signproperty}.

Similar to Definition \ref{def:contcmm}, we introduce the following.
\begin{definition}
A sequence $a=(a_0, a_1, \cdots)$ is complementary monotone (CMM) if it is nonnegative, nonincreasing, and its complementary kernel $a^c$ satisfying $a*a^c=(1,1,\cdots)$ is also nonnegative and nonincreasing.
\end{definition}

\begin{remark}\label{rmk:CMseq}
We note a characterization of the convolution inverse of a completely monotone (CM) sequence in \cite{liliu2018} (or  \cite{li2021complete} for an improved version). A sequence $v=(v_0, v_1, \ldots)$ is said to be CM if $((I-E)^j v)_k\ge 0,~\text{ for any } j\ge 0, k\ge 0$, where $(Ev)_j=v_{j+1}$.  In \cite{liliu2018},  it has been shown that the inverse of a CM sequence $a$ can be written as $a^{(-1)}=(a_0^{-1}, -c_1, -c_2, \cdots)$ with $(c_1, c_2, \cdots)$ being CM.  Consequently, consider the complementary $a^c$ in the sense that $a*a^c=(1,1,1, \cdots)$. Then, $a^c=a^{(-1)}*(1,1,\cdots)$ and $a^c$ is also CM. This is in fact our original motivation to consider using complementary kernel to inverstigate the monotonicity preserving properties, before we noticed \cite[Theorem 2.2]{clement1981asymptotic}.
\end{remark}

Below, we will show that the CMM property is simply the completely positivity plus the nonincreasing property. 
\begin{theorem}\label{thm:uniformcmm}
The following are equivalent:
\begin{enumerate}[(a)]
\item The sequence $a=(a_0, a_1,\cdots)$ is CMM.

\item $a_0\neq 0$ and $a=(a_0, a_1,\cdots)$  is nonincreasing and completely positive.

\item 
The sequence $a=(a_0, a_1, \cdots)$ is nonincreasing with $a_0\neq 0$, and its convolutional inverse $b=a^{(-1)}=(b_0, b_1, \cdots)$ satisfies
\begin{gather}\label{eq:signconditionb}
b_0>0, \quad b_j\le 0, \quad \forall j\ge 1.
\end{gather}

\item The complementary kernel of $a$ is nonincreasing and completely positive.
\end{enumerate}
\end{theorem}

\begin{proof}

(c) $\Rightarrow$ (a): By the relations $a_0b_0=1$ and
\[
a_nb_0=-\sum_{j=1}^n a_{n-j}b_j, \quad n\ge 1,
\]
it is straightforward to see that $a$ must be nonnegative if $b_0>0$ and $b_j\le 0$ for $j\ge 1$ by induction.  See Lemma \ref{lmm:signofentry} below for the more general version on nonuniform meshes. By $a^c=b*(1,1,\cdots)$, $a^c$ is nonincreasing. Besides, since $a$ is nonincreasing, the first element in $(a^c)^{(-1)}=a*(1, -1, 0, \cdots)$ is positve and other elements are nonpositive. By the result just proved, $a^c$ is nonnegative, or $\sum_{j=0}^n b_j \ge 0$. 

(a) $\Rightarrow (b)$:  That $a_0>0$ and $a$ is nonincreasing are clear. Since $a^c$ is nonnegative and nonincreasing, \eqref{eq:signproperty} holds. Lemma \ref{lmm:discretecpuni} then gives the result.

(b) $\Rightarrow (c)$:  This follows directly by Lemma \ref{lmm:discretecpuni}.

The complementary kernel of $a$ exists if and only if $a_0\neq 0$. Since $a$ is CMM if and only if the complementary kernel $a^c$ is CMM, then the equivalence between $(d)$ and $(a)$ is then clear as we have established the equivalence between $(a)$ and $(c)$.
\end{proof}

Theorem \ref{thm:uniformcmm} indicates that the CMM property is just nonincreasing plus the complete positivity. As we can see from the equivalence between (a) and (c), the nonincreasing property somehow implies the nonnegativity of $a^c$ (or $\sum_{j=0}^n b_j \ge 0$).  The following tells us that the CMM property is weaker than the log-convexity.
\begin{lemma}\label{lmm:cmmuni}
 If $a=(a_0, a_1,\cdots)$ is nonnegative, nonincreasing and is log-convex in the sense $a_{j-1}a_{j+1}\ge a_j^2$, then it is CMM.
\end{lemma}
We refer the readers to \cite[Lemma 2.3]{liao2020positive} for the result on the signs of the inverse if the discrete kernel is log-convex.
We remark that for $a$ to be CMM, $b_2\le 0$ is equivalent to $a_0a_2\ge a_1^2$. Nevertheless, the log-convexity for all $j$ is clearly strong.

\section{Complementary monotone kernels on nonuniform meshes}\label{sec:Rcmmnonuniform}

We then generalize the CMM properties mentioned in section \ref{sec:cmmunifom} to the the nonuniform meshes. This will be the main tool we use in this paper to prove the monotonicity preserving properties on nonuniform meshes.

\subsection{Pseudo-convolution}\label{subsec:pc}

Let us have brief review of the pseudo-convolution discussed in \cite{fengli2023a}. We arrange the kernel $\{a_{n-j}^n\}$ into a lower triangular array $A$ of the following form
\begin{gather}\label{eq:arraykernel}
A=\begin{bmatrix}
a_0^{1} &  &  &  &  \\
a_1^{2}& a_0^{2} & & & \\
\cdots & \vdots & \vdots &  &  \\
a_{n-1}^{n} &  \cdots & a_1^{n} & a_0^{n} &\\
\cdots & \vdots & \vdots &   & \vdots\\
\end{bmatrix}.
\end{gather}
The pseudo-convolution between two such kernels $A$ and $B$ is defined to be another kernel $C:=A\pc B$, given by
\begin{gather}\label{eq:convnonuni}
c_{k}^n=\sum_{j=0}^k a_{k-j}^n b_j^{n+j-k}, \quad \text{or}\quad c_{n-k}^n=\sum_{j=k}^n a_{n-j}^n b_{j-k}^j.
\end{gather}
If $a_{j}^n=a_{j}$ and $b_{j}^n=b_j$ are both  independent of $n$, then  it reduces to the usual convolution.  By the definition, the convolution for $n\le N$ does not depend on the data with $n>N$. Hence, though the discussion here is for infinite arrays, the result  can apply to array kernels with finite data.

Consider the following special kernels
\begin{gather}\label{eq:specialkernel}
I=
\begin{bmatrix}
1 &  &    &  &\\
& 1 &  &  &\\
 &    & \vdots & & \\
& &    &1 & \\
 &  &    & & \vdots\\
\end{bmatrix},
\quad
L=
\begin{bmatrix}
1 &  &  &  &  \\
1& 1 & & &  \\
 \vdots & \vdots &  \vdots&  &  \\
1&1 & \cdots & 1 &  \\
 \vdots&\vdots  &\vdots  &\vdots  &\vdots  \\
\end{bmatrix},
\quad L^{(-1)}=
\begin{bmatrix}
1 &  &  &  &  \\
-1& 1 & & &  \\
 & -1 &  1&  &  \\
& & \vdots & 1 &  \\
&  &  &\cdots  &\vdots  \\
\end{bmatrix}.
\end{gather}
It can be verified that $I$ is the identify for the pseudo-convolution, and the following properties hold.
\begin{enumerate}[(a)]
\item The pseudo-convolution is associative.
\item For a given $A$, if a kernel $B$ satisfies $A\pc B=I$, then $B\pc A=I$.
\end{enumerate}
The kernel $B$ is actually the ROC kernel defined in \cite{liao2021analysis}.  Clearly, $B$ is both the left inverse and the right inverse of $A$ for pseudo-convolution so we may simply call it the inverse, and denote
\[
A^{(-1)}:=B, \quad \text{such that}~B\pc A=A\pc B=I.
\]
The following lemma from \cite[Lemma 4.3]{fengli2023a} is reminiscent of the M-matrices and is often useful.
\begin{lemma}\label{lmm:signofentry}
Let $B$ be the inverse of $A$. If $B$ has positive diagonal elements and nonpositive off-diagonal elements, then $A$ has nonnegative elements and the entries on the diagonal are positive. 
\end{lemma}

The kernel $L$ corresponds to the sequence $(1,1,\cdots)$ in the usual convolution, and $L^{(-1)}$ in \eqref{eq:specialkernel} is clearly the inverse of $L$. With this, one may define the complementary kernels. 
\begin{definition}
For a given $A$, the kernel $C_R$ with $A\pc C_R=L$ is called the right complementary kernel. The kernel $C_L$ with $C_L\pc A=L$ is called the left complementary kernel.
\end{definition}
If $A$ is invertible, then direct verification tells us that $C_R=A^{(-1)}\pc L$ and $C_L=L\pc A^{(-1)}$.
Using this fact, one has $C_R^{(-1)}=L^{(-1)}\pc A$ and $C_L^{(-1)}=A\pc L^{(-1)}$. Consequently, one has the following observation.
\begin{lemma}
Moreover, $a_j^n$ is nonincreasing in $n$ if and only if the inverse of $C_R$ has nonpositive off-diagonals; $a_j^n$ is nonincreasing in $j$ if and only if the inverse of $C_L$ has nonpositive off-diagonals.
\end{lemma}

One can also define the pseudo-convolution between a kernel and a vector.
Consider
\[
V=\{x=(x_1, x_2,\cdots)^T:\quad x_i\in \R \}.
\]
 Define the pseudo-convolution $\pc$: $K\times V\to V$, $y=A\pc x$
by
\begin{gather}
y_{n}=\sum_{j=1}^{n}a_{n-j}^n x_{j}.
\end{gather}
then, it holds that
$A\pc(B\pc x)=(A\pc B)\pc x$.

\subsection{Basic definitions and facts}

For the array kernels, the monotonicity of the kernels is not very straightforward now. We need to look at the columns and rows. 
\begin{definition}
Consider an array kernel $A=(a_{n-j}^n)$.  We call $A$ column monotone if it has nonnegative entries and
$a_{j-1}^{n-1}\ge a_j^n$.
We call it to be row monotone, if it has nonnegative entries and
$a_{j-1}^n\ge a_j^n$.
We call it doubly monotone if it is both column monotone and row monotone.
\end{definition}
The column monotonicity actually means that for different time $n$, the approximation of the kernel $a$ on a fixed interval $I_j=(t_{j-1}, t_j)$ is nonincreasing. The row monotonicity means that for a fixed time $n$, the approximation of the kernel $a$ is monotone over different intervals $I_j$. 

As  a generalization of the uniform mesh case, we propose the following.
\begin{definition}\label{def:quasicm}
\begin{enumerate}[(a)]
\item A column monotone kernel $A$ is called right complementary monotone (R-CMM) if its right complementary kernel $C_R$ is doubly monotone. 
\item A row monotone kernel is called left complementary monotone (L-CMM) if its left complementary kernel is doubly monotone.
\item A doubly monotone kernel $A$ is complementary monotone (CMM) if it is both R-CMM and L-CMM.
\end{enumerate}
\end{definition}

We have seen that the signs of the entries of the inverse can also be used to characterize the CMM property for uniform meshes, but this is not the case for nonuniform meshes. If $A^{(-1)}$ has positive diagonal and nonpositve off-diagonals, then $C_R$ is row monotone but cannot ensure the column monotonicity of $C_R$ so that relation \eqref{eq:rowsums} holds, which is needed in section \ref{subsec:monpreserving}. Hence, we use the complementary kernel to define the R-CMM property here.  Note that we are not requiring the kernel $A$ itself to be doubly monotone because the row monotonicity is not needed.

The pseudo-convolution for $n\le N$ is not affected by the data with $n>N$. Hence, one may consider the local versions of the CMM concepts.
 \begin{definition}
 If $A$ is column monotone for $n\le N$ and $C_R$ is doubly monotone for all $n\le N$,
 then we call $A$ to be ``local R-CMM with range $N$''. The local L-CMM and local CMM are similarly defined.
 \end{definition}
 
The results below are mainly stated for R-CMM kernels, while the ones for L-CMM can be proved similarly. 
Moreover, we only study the global CMM properties and the local versions can be easily obtained by the local feature of the pseudo-convolution.

We now give some characterizations of the R-CMM kernels.
\begin{theorem}\label{thm:qcmchar}
The following are equivalent. 
\begin{enumerate}[(a)]
\item The array kernel $A$ is R-CMM;
\item The right complementary kernel $C_R$ is doubly monotone and $C_R^{(-1)}$ has positive diagonals and nonpositive off-diagonals;
\item  $A$ is column monotone, and both $A^{(-1)}$ and $(L^{(-1)}\pc A\pc L)^{(-1)}$ have positive diagonals and nonpositive off-diagonals.
\end{enumerate}
\end{theorem}
\begin{proof}

The equivalence between (a) and (b) follows from Definition \ref{def:quasicm}  and the fact $C_R^{(-1)}=L^{(-1)}\pc A$.
From (b) to (c), one only has to use the fact $C_R^{(-1)}=L^{(-1)}\pc A$ and apply Lemma \ref{lmm:signofentry}. From (c) to (b), one apply directly the observation $A^{(-1)}=C_R\pc L^{(-1)}$
and $(L^{(-1)}\pc A\pc L)^{(-1)}=L^{(-1)}\pc C_R$. We omit the details.
\end{proof}

Clearly, the column monotonicity of $A$ is equivalent to the nonpositivity of off-diagonals in $C_R^{(-1)}$. This, together with the positive diagonals, implies that $C_R$ must have nonnegative elements. We remark, however, the column monotonicity of $A$ is clearly stronger than the nonnegativity of the elements of $C_R$.  

Below, we present some necessary conditions hidden in Theorem \ref{thm:qcmchar} above.
\begin{corollary}\label{cor:Rcmmproperties}
Suppose $A$ is R-CMM. Then, the following facts hold.
\begin{enumerate}[(1)]
\item The kernel $L^{(-1)}\pc A\pc L$ has nonnegative entries, and it implies that for each $n\ge 1$,
\begin{gather}\label{eq:rccmon}
\sum_{j=k}^{n+1}a_{n+1-j}^{n+1} \ge \sum_{j=k}^{n}a_{n-j}^{n}, \quad 
1\le k\le n-1.
\end{gather}
Moreover, it is row monotone.

\item Let $A^{(-1)}=B=(b_{n-j}^n)$. Then It holds that
\begin{gather}\label{eq:sumb}
b_0^n>0, \quad b_{n-j}^n\le0, \quad \sum_{j=1}^nb_{n-j}^n\ge 0, \forall n\ge 1\quad j<n.
\end{gather}
\end{enumerate}
\end{corollary}
The claims in \eqref{eq:sumb} are just reinterpretation of the signs for the elements of complementary kernels.  The condition \eqref{eq:rccmon} is actually very natural, since 
\[
 \sum_{j=k}^{n}a_{n-j}^{n} \approx
\int_0^{\sum_{j=k}^n\tau_j}a(s)\,ds.
\]

Below, we give a sufficient condition for $A$ to be R-CMM. 
We recall a basic result in \cite[Lemma 2.3]{liao2020positive}:
\begin{lemma}
If a kernel $\tilde{A}=(\tilde{a}_{n-j}^n)$ has positive entries such that 
\begin{gather}\label{eq:logconv}
\tilde{a}_{j-1}^{n-1} \tilde{a}_{j+1}^{n} \ge \tilde{a}_j^{n} \tilde{a}_j^{n-1},
\end{gather}
then the inverse has positive diagonal elements and nonpositive off-diagonal elements. If moreover, $\tilde{A}$ is column monotone, then the right complementary kernel $\tilde{C}_R$ has nonnegative elements so that $\tilde{C}_R$ is row monotone. 
\end{lemma}
Note that the statement here is slight different from that in \cite[Lemma 2.3]{liao2020positive}.
In \cite[Lemma 2.3]{liao2020positive}, the conditions of column monotonicity and \eqref{eq:logconv} are proposed together. However, if we go over the proof, one can find that the column monotonicity is used for the signs of RCC kernels, namely the last inequality in \eqref{eq:sumb}. Moreover, in \cite[Lemma 2.3]{liao2020positive}, they assumed strict monotonicity along columns, which is not needed by us.

As mentioned, we need $C_R$ to be column monotone. The condition
\eqref{eq:logconv} imposed on $A$ seems not enough for $A$ to be R-CMM.  We need to put conditions on $L^{(-1)}\pc A\pc L$ as well. Hence, a sufficient condition would be the following.
\begin{proposition}\label{pro:suffcond}
If the following conditions are satisfied:
\begin{enumerate}[(1)]
\item $A$ is column monotone;
\item the kernels $A$ and $L^{(-1)}\pc A\pc L$ both have positive elements and both satisfy \eqref{eq:logconv};
\end{enumerate}
then $A$ is R-CMM.
\end{proposition}
Note that we are not requiring the column monotonicity for $L^{(-1)}\pc A\pc L$, which may not hold for schemes considered. We believe that the conditions in Proposition \ref{pro:suffcond} are kind of strong in general. Nevertheless, we will use Proposition \ref{pro:suffcond} for the example in section \ref{sec:ex}.

The R-CMM could be preserved under a certain scaling transform. In fact, we have the following. 
\begin{lemma}
Suppose $A$ is R-CMM and $\tau$ is diagonal as
\begin{gather}\label{eq:tau}
\tau=
\begin{bmatrix}
\tau_1 &  &    &  &\\
& \tau_2 &  &  &\\
 &    & \vdots & & \\
& &    &\tau_j & \\
 &  &    & & \vdots\\
\end{bmatrix},
\end{gather}
with the property that $0<\tau_1\le \tau_2\le \tau_3\cdots$. Then, $A\pc \tau$ is R-CMM.
\end{lemma}
\begin{proof}
For any diagonal kernel with $\tau_i>0$, $A\pc \tau$ is again column monotone.
Moreover, $(A\pc \tau)^{(-1)}=\tau^{(-1)}\pc A^{(-1)}$. Hence, the signs of the elements in inverse of $A\pc \tau$ are as desired. Moreover, 
\[
(L^{(-1)}\pc A\pc\tau\pc L)^{(-1)}=L^{(-1)}\pc \tau^{(-1)}\pc A^{(-1)}\pc L
=L^{(-1)}\pc \tau^{(-1)}\pc C_R.
\]
Note that $C_R$ is doubly monotone. Clearly, if $\tau_1\le \tau_2\le \tau_3\cdots$,
then $\tau^{(-1)}\pc C_R$ is column monotone. Then, $(L^{(-1)}\pc A\pc\tau\pc L)^{(-1)}$ has the desired signs. The result then follows from Theorem \ref{thm:qcmchar}.
\end{proof}

Note that, under this scaling, \eqref{eq:logconv} and the column monotonicity of $A$ are invariant (the doubly monotonicity is hard to preserve under the scaling though). Hence, the property R-CMM is kind of robust under such a scaling. (The condition \eqref{eq:logconv} for $L^{(-1)}\pc A\pc L$ may be broken, though.)

\subsection{R-CMM versus complete positivity}

Next, we consider the resolvent kernels for nonuniform meshes using the pseudo-convolution:
\begin{gather}\label{eq:disreldef}
R_{\lambda}+\lambda R_{\lambda}\pc A=\lambda A
\Longleftrightarrow
A-R_{\lambda}\pc A=\frac{1}{\lambda}R_{\lambda}.
\end{gather}
\begin{lemma}\label{lmm:Rbasics}
Suppose the diagonal elements of $A$ are positive and its right complementary kernel is $C_R$. Then, the resolvent $R_{\lambda}$ defined by \eqref{eq:disreldef} always exists for $\lambda>0$.
Moreover, the following holds:
\begin{enumerate}[(a)]
\item $R_{\lambda}\pc A=A\pc R_{\lambda}$, $R_{\lambda}\pc A^{(-1)}=A^{(-1)}\pc R_{\lambda}$;
\item $I-R_{\lambda}=(I+\lambda A)^{(-1)}=\lambda^{-1} R_{\lambda}\pc A^{(-1)}$;
\end{enumerate}
\end{lemma}
The following describes the asymptotic behavior of the resolvents, which could be insightful and is needed in section \ref{subsec:monpreserving} and has been proved in \cite{fengli2023a}.
\begin{lemma}\label{lmm:resolvasym}
Suppose that $A$ is invertible. The resolvent $R_{\lambda}$ satisfies the following as $\lambda\to \infty$:
\[
R_{\lambda}=I-\lambda^{-1}A^{(-1)}+O(\lambda^{-2}).
\]
The $O(\lambda^{-2})$ is elementwise under the limit $\lambda\to +\infty$.
\end{lemma}

Similar to Definition \ref{def:uniformcp}, the complete positivity was introduced in \cite{fengli2023a}.
\begin{definition}
An array kernel $A$ is completely positive if $0<(R_{\lambda})_0^n<1$, $(R_{\lambda})_{n-j}^n\ge 0$ and $\sum_{j=1}^n (R_{\lambda})_{n-j}^n\le 1$ for all $\lambda>0$. 
\end{definition}
It has been shown in \cite{fengli2023a} that 
\begin{lemma}\label{lmm:resolventsigns}
Let $A$ be an invertible array kernel and $B=A^{(-1)}=(b_{n-j}^n)$.
Then, $A$ is completely positive if and only if the conditions \eqref{eq:sumb} hold. 
\end{lemma}

By Corollary \ref{cor:Rcmmproperties}, $A$ being R-CMM is clearly stronger than being completely positive. In fact, the column monotonicity of $A$ implies that $C_R$ is nonnegative and thus $\sum_{j=1}^n (R_{\lambda})_{n-j}^n\le 1$. However, $C_R$ being nonnegative cannot imply that $A$ is column montonoe. Hence, R-CMM is stronger.  

Similar to the uniform mesh version in Theorem \ref{thm:uniformcmm}, one has
\begin{theorem}\label{thm:Rdoublmon}
The following are equivalent.
\begin{enumerate}[(a)]
\item The array kernel $A$ is R-CMM;

\item The kernel  $A$ is column monotone with positive diagonals, and both the resolvents of $A$ and $L^{(-1)}\pc A\pc L$ defined in \eqref{eq:disreldef} have nonnegative entries for all $\lambda>0$;

\item  The diagonals of $A$ are positive and $I+\lambda A$ is R-CMM for all $\lambda>0$.
 \end{enumerate}
\end{theorem}
\begin{proof}
(a) $\Rightarrow$ (b): 
Let $B=A^{(-1)}$. Then $B$ has positive diagonals and nonpositive off-diagonals by the R-CMM property. Note that $R_{\lambda}^{(-1)}= I+\lambda^{-1} A^{(-1)}$ has nonpositive off-diagonals and positive diagonals. Then, the entries of $R_{\lambda}$ are nonnegative. 

Let $M=\lambda L^{(-1)}\pc A\pc L$ whose inverse has nonpositive off-diagonals by the R-CMM property of $A$. 
Writing $(I+M)^{(-1)}=I-N$, then $N$ is the resolvent of $L^{(-1)}\pc A\pc L$. One can similarly find $N^{(-1)}=I+M^{(-1)}$.
Hence, $N$ has nonnegative entries by Lemma \ref{lmm:signofentry}.

(b) $\Rightarrow$ (c):  clearly, $I+\lambda A$ is column monotone. Note that $(I+\lambda A)^{(-1)}=I-R_{\lambda}$ has nonpositive off-diagonals.
Consider $L^{(-1)}\pc(I+\lambda A)\pc L=I+\lambda L^{(-1)}\pc A\pc L$.  The inverse of $L^{(-1)}\pc(I+\lambda A)\pc L$ is thus $I$ minus the resolvent of $L^{(-1)}\pc A\pc L$, and thus has nonpositive off-diagonal entries. Theorem \ref{thm:qcmchar} implies that $I+\lambda A$ is R-CMM.

(c) $\Rightarrow$ (a): the condition implies that $\lambda^{-1}I+A$ is R-CMM for all $\lambda>0$. Taking $\lambda\to \infty$, one can show that the complementary kernel is continuous in $\lambda^{-1}$ elementwise.
The monotonicity and nonnegativity is preserved in the limit.
\end{proof}

We have the following observation.
\begin{corollary}
If $A$ is R-CMM, then it holds that
\begin{gather}\label{eq:rowsums}
\sum_{j=k}^n (R_{\lambda})_{n-j}^n\ge \sum_{j=k}^{n-1}(R_{\lambda})_{n-1-j}^{n-1}, \forall 1\le k\le n-1.
\end{gather}
\end{corollary}

\begin{proof}
 Since $I+\lambda A$ is also R-CMM, then $(I-R_{\lambda})\pc L$, as the complementary kernel of $I+\lambda A$, is doubly monotone.  Hence, $R_{\lambda}\pc L$ are nondecreasing along the columns and thus \eqref{eq:rowsums} holds.
\end{proof}
\begin{remark}
Note that the right complementary kernel of $R_{\lambda}$ is $\lambda^{-1}C_R+L$ which is doubly monotone. If we can show that $R_{\lambda}$ is column monotone, then $R_{\lambda}$ is R-CMM, which is left for future study.
\end{remark}

The relation \eqref{eq:rowsums}  is an important property we need in section \ref{subsec:monpreserving} to show that the numerical solutions to autonomous equation on nonuniform meshes are monotone. This relation is very natural on uniform meshes and can be derived by complete positivity as $r_{\lambda}$ is nonnegative. For nonuniform meshes, it is difficult to obtain using only complete positivity. The R-CMM property is slightly stronger but seems more suitable as it has built in the fading memory principle.

\section{Monotonicity properties for the numerical solutions}\label{sec:monpre}

\subsection{Monotonicity properties of the R-CMM schemes}\label{subsec:monpreserving}

In this section, we consider  the numerical scheme \eqref{eq:integraldis}.
In other words
\begin{gather}
u=h+ \bar{A}\pc \tau \pc f=h+A \pc f,
\end{gather}
where $\tau$ is given in \eqref{eq:tau}.
We will consider the schemes where $A$ is R-CMM and investigate the two monotonicity properties mentioned in section \ref{subsec:timecontinuous} in the discrete level.

\begin{remark}
A discretization of the fractional ODE may be performed for the differential form \eqref{eq:fracode} and one may obtain
\begin{gather*}
C_R\pc(\nabla_{\tau}u)=(B \pc (u-u_0))_n=f(t_n, u_n), \quad \nabla_{\tau}u:=L^{(-1)}\pc (u-u_0).
\end{gather*}
Then, this can be converted into
\[
u_n=u_0+A\pc f(t_i, u_i)=u_0+\bar{A}\pc \tau \pc f
,\quad A=B^{(-1)}.
\]
For  \eqref{eq:fracode}, the elements in $C_R=(c_{n-j}^n)$ can be regarded as the {\it average}
of $(t_n-s)^{-\alpha}/\Gamma(1-\alpha)$ on $[t_{j-1}, t_j]$ while $a_{n-j}^n$ is the {\it integral} of $(t_n-s)^{\alpha-1}/\Gamma(\alpha)$ on this interval. In fact, the scaling of $c_0^n$ is like $\tau_n^{-\alpha}$
so that the scaling of $a_0^n$ is like $\tau_n^{\alpha}$ which agrees with the scaling for integral.
With this relation in hand, the differential scheme is then said to be R-CMM if $A$ is R-CMM.
\end{remark}

In this section, we prove the monotonicity preserving properties of the numerical solutions with R-CMM kernels. 
\begin{theorem}\label{thm:nummon1}
Suppose $f(t, \cdot)$ is uniformly Lipschitz continuous on $[0, T]$ with Lipschitz constant $M>0$. Assume $A$ is local R-CMM with range $N$. Assume $M\sup_{j\le N} a_0^j=M\sup_{j\le N}\tau_j \bar{a}_0^j<1$.   If $\gamma(t_n):=h^{(1)}(t_n)-h^{(2)}(t_n)$ satisfies $r=\gamma-R_{\lambda}\pc \gamma \ge 0$ where $R_{\lambda}$ is defined in \eqref{eq:disreldef},  then the corresponding numerical solutions satisfy $u^{(1)}_n \ge u^{(2)}_n$ for all $n\le N$. In particular, if $h^{(i)}$'s are two constants and $h^{(1)}\ge h^{(2)}$, then  the claim holds.
\end{theorem}

\begin{proof}
Clearly, we can consider R-CMM kernels without loss of generality.  Basically, if $w_n=u^{(1)}_n-u^{(2)}_n$, then
\[
w_n=\gamma(t_n)+\sum_{j=1}^{n}a_{n-j}^n g_i w_i, \quad g_i=\int_0^1\partial_u f(t_i, z u_i^{(1)}
+(1-z)u_i^{(2)})\,dz.
\]
This is just  $w=\gamma+A\pc  (gw)$ where $gw=(g_1w_1, g_2w_2,\cdots)$. Taking pseudo-convolution with $I-R_{\lambda}$ on the left, Lemma \ref{lmm:Rbasics} then gives
\begin{gather}\label{eq:diffaux1}
w_n=r_n+R_{\lambda}*(w + gw/\lambda)_n,
\end{gather}
where $r=\gamma-R_{\lambda}\pc \gamma \ge 0$. 
Choosing $\lambda$ large enough, $1+g_n/\lambda>0$ for all $n\le N$. Then,
\begin{gather}\label{eq:diffaux2}
(1-(R_{\lambda})_0^n-(R_{\lambda})_0^n g_n/\lambda)w_n
=r_n+\sum_{j=1}^{n-1} (R_{\lambda})_{n-j}^n w_j(1+g_j/\lambda).
\end{gather}
By Lemma \ref{lmm:resolvasym}, 
\[
1-(R_{\lambda})_0^{n}+(R_{\lambda})_0^{n}\lambda^{-1}g_n
=\lambda^{-1}(a_0^{n})^{-1}+(1-\lambda^{-1}(a_0^{n})^{-1})\lambda^{-1}g_n+O(\lambda^{-2}).
\]
Since $|g_n|\le M$, it follows that the coefficient of $w_n$ in \eqref{eq:diffaux2} is positive if $\lambda$ is large enough by the condition $M\sup_{j\le N} a_0^j<1$. Since $R_{\lambda}\ge 0$ by Theorem \ref{thm:Rdoublmon}, the result then follows by simple induction.

If $h^{(i)}$'s are two constants with $h^{(1)}-h^{(2)}\ge 0$, using Lemma \ref{lmm:resolventsigns}, one can conclude that $r\ge 0$. The result then follows.
\end{proof}

Below, we discuss the monotonicity of the solutions.
Here, we only consider $h(t)\equiv u_0$ and the autonomous cases. 

\begin{theorem}\label{thm:nummon2}
Assume $h(t)\equiv u_0$ and $f(t, u)\equiv f(u)$. Suppose $f$ is Lipschitz continuous with
$|f'(u)|\le M$ for some $M>0$. If $A$ is local R-CMM with range $N$ for the numerical method, then for $M\sup_{j\le N}a_0^j=M\sup_{j\le N}\tau_j \bar{a}_0^j<1$, the solution $\{u_n\}$ is monotone for all $0\le n\le N$.
\end{theorem}

\begin{proof}
If $f(u_0)=0$, one can see that the solution is always $u_n\equiv u_0$ (note that the numerical solution is uniquely solvable).

Below, we only focus on $f(u_0)>0$, as the discussion for $f(u_0)<0$ is similar. Again, we consider the R-CMM sequences without loss of generality (for the local R-CMM ones, one only needs to repeat the argument for $n\le N$).

Consider first $n=1$. Then,
\[
u_1-a_0^1 f(u_1)=u_1-\tau_1 \bar{a}_0^1 f(u_1)=u_0.
\]
Let $\mu(u)=u-\tau_1 \bar{a}_0^1 f(u)$. Clearly, $\mu(u_0)<u_0=\mu(u_1)$ and the function is increasing. Hence, one has
$u_1>u_0$.
Now, define
\[
v_n:=u_{n+1}-u_n, \quad n\ge 0.
\]
Hence, $v_0>0$ and we aim to show $v_n\ge 0$ for $n\ge 1$.

Recall $A$ is R-CMM, and $R_{\lambda}$ is its resolvent. Let $u-u_0:=(u_1-u_0, u_2-u_0, \cdots)$ and $f(u)=(f(u_1), f(u_2), \cdots)$. Then, $u-u_0=A\pc f(u)$. Taking pseudo-convolution with $I-R_{\lambda}$ on the left, one has
\[
u_n-u_0=R_{\lambda}\pc (u-u_0+f(u)/\lambda)_n,
\]
which gives
\begin{multline*}
v_n=(R_{\lambda})_{n}^{n+1}[u_1-u_0+f(u_1)/\lambda]
+\sum_{j=1}^n(R_{\lambda})_{n-j}^{n+1}[u_{j+1}-u_0+\frac{f(u_{j+1})}{\lambda}]\\
-\sum_{j=1}^n(R_{\lambda})_{n-j}^n[u_j-u_0+\frac{f(u_j)}{\lambda}]=: I_1+I_2+I_3.
\end{multline*}

Mimicking the continuous case, one would arrange it into the following
\begin{gather}\label{eq:I2plusI3}
I_2+I_3=\sum_{j=1}^n[(R_{\lambda})_{n-j}^{n+1}-(R_{\lambda})_{n-j}^n]\left[u_{j+1}-u_0+\frac{f(u_{j+1})}{\lambda}\right]
+\sum_{j=1}^n (R_{\lambda})_{n-j}^n(v_j+\lambda^{-1}g_j v_j),
\end{gather}
where $g_j=\int_0^1 f'(z u_{j+1}+(1-z)u_j)\,dz$.
The issue is that the sign of the first term is not easy to determine. Hence, we must rearrange the terms to resolve this. We rewrite
\[
u_{j+1}-u_0=u_1-u_0+\sum_{\ell=1}^j v_{\ell}, \quad f(u_{j+1})=f(u_1)+\sum_{\ell=1}^j g_{\ell}v_{\ell}.
\]
Introduce the notation
\[
S_j^n:=\sum_{\ell=j}^n (R_{\lambda})_{n-\ell}^{n+1}-\sum_{\ell=j+1}^{n}
(R_{\lambda})_{n-\ell}^n.
\]
By relation \eqref{eq:rowsums},  $S_j^n\ge 0$. Using $S_j^n$, one then has
\begin{gather*}
I_1+I_2+I_3=S_0^n(u_1-u_0+f(u_1)/\lambda)
+\sum_{j=1}^{n-1} v_j(1+g_j/\lambda) S_j^n
+v_n(1+g_n/\lambda) (R_{\lambda})_0^{n+1},
\end{gather*}
so that
\begin{gather}
v_n(1-(R_{\lambda})_0^{n+1}-(R_{\lambda})_0^{n+1}\lambda^{-1}g_n)
=\sum_{j=1}^{n-1} v_j(1+g_j/\lambda) S_j^n+\gamma_n,
\end{gather}
where
\[
\gamma_n=S_0^n(u_1-u_0+f(u_1)/\lambda).
\]
Hence, for $\lambda$ sufficiently large $u_1-u_0+f(u_1)/\lambda>0$ and thus $\gamma_n\ge 0$.

Now, consider the coefficient of $v_n$. By Lemma \ref{lmm:resolvasym},
\[
\epsilon_n:=1-(R_{\lambda})_0^{n+1}+(R_{\lambda})_0^{n+1}\lambda^{-1}g_n
=\lambda^{-1}(a_0^{n+1})^{-1}+(1-\lambda^{-1}(a_0^{n+1})^{-1})\lambda^{-1}g_n+O(\lambda^{-2}).
\]
Since $|g_n|\le M$ and $a_0^{n+1}=\bar{a}_0^{n+1}\tau_{n+1}$, if
$M a_0^{n+1}<1$,
then $\epsilon_n>0$ for $\lambda$ large enough.

For each $n$, one may choose a suitably large $\lambda$ (depending on $n$) so that the coefficients are positive to find $v_n\ge 0$, by induction.
\end{proof}

\begin{remark}\label{rmk:cmtuniform}
For uniform meshes, the first term in \eqref{eq:I2plusI3} would vanish and the monotonicity of the numerical solution follows easily by the nonnegativity of $R_{\lambda}$. Hence the complete positivity is enough for this monotonicity property in the case of uniform meshes. 
\end{remark}

\begin{remark}
For time fractional ODEs, $\bar{a}_0^n\sim \frac{1}{\Gamma(\alpha+1)}\tau_n^{\alpha-1}$.
Then, the condition on the stepsize in Theorem \ref{thm:nummon2} agrees with the one in \cite{li2021complete}.
\end{remark}

\subsection{Discussion}

We perform some discussion on the  time fractional differential equations here as they are an important class of theVolterra integral equations we considered. The kernels for time fractional ODEs are clearly CM, which is thus log-convex and CMM.  

 The so-called CM-preserving schemes on uniform meshes were proposed in \cite{li2021complete}.
 These schemes are on the uniform meshes and require the discrete kernel $\{a_j\}$ to be a CM sequence.  It has been shown that the CM-preserving methods have many good stability properties and have been used sucessfully to prove some sharp estimates \cite{chen2022using}. The standard $L1$ scheme, the Gr\"unwald-Letnikov method, and the method with averaged kernel are CM-preserving for the time fractional ODEs \cite{li2021complete}.   The CM-preserving schemes are clearly CMM by Remark \ref{rmk:CMseq}. However,  such methods may be restricted in applications, as they are restricted to uniform meshes and high order schemes often break the CM property.

Another option is to consider discretization that are log-convex, which is related to the positive definiteness of the methods and may enjoy some good properties \cite{liao2020positive}.
As indicated in Lemma \ref{lmm:cmmuni},  the monotonicity, nonnegativity and the log-convexity $a_{j+1}a_{j-1}\ge a_j^2$ will suffice for the CMM property on uniform meshes. This means that the CMM property is in fact kind of weak for uniform meshes.

As for the nonuniform meshes, the result in Proposition \ref{pro:fracqcm} for the time fractional ODE is interesting in the sense that there is no restriction on the ratios $\tau_{j+1}/\tau_j$ for the stepsizes. This may indicate that the R-CMM property is quite flexible in practice.

\section{Application to fractional ODEs}\label{sec:ex}

In this subsection we consider the integral equation \eqref{eq:fracint} for the fractional ODEs with $\alpha\in (0,1)$.
We consider the simplest scheme for the integral on the nonuniform meshes. In particular, we consider the method
\begin{gather}\label{eq:fracdis1}
u_n=u_0+\sum_{j=1}^n a_{n-j}^n f(t_j, u_j),
\end{gather}
with the coefficients given by
\begin{gather}\label{eq:fraccoe}
a_{n-j}^n=\frac{1}{\Gamma(\alpha)}\int_{t_{j-1}}^{t_j}(t_n-s)^{\alpha-1}\,ds.
\end{gather}
That means we approximate $f(s, u(\cdot))$ using the constant interpolation with the right point
on each time interval. Such an integral averaged method on uniform meshes has been applied to investigate the time continuous fractional gradient flows and the fractional SDEs in \cite{li2019discretization}.

\subsection{The R-CMM property of the averaged integral scheme}

\begin{proposition}\label{pro:fracqcm}
The kernel $A$ in the scheme \eqref{eq:fracdis1}-\eqref{eq:fraccoe} for the fractional ODE is R-CMM. Consequently, the results in Theorem \ref{thm:nummon1} and Theorem \ref{thm:nummon2} hold for this method.
\end{proposition}

\begin{proof}
We basically apply Proposition \ref{pro:suffcond} to show that $A$ is R-CMM. Here, we verify the conditions. Let $A=(a_{n-j}^n)$. It is clear that
\[
a_{n-j}^n=\frac{1}{\Gamma(\alpha+1)}((t_n-t_{j-1})^{\alpha}
-(t_n-t_j)^{\alpha}).
\]
Clearly, $A$ is strictly column monotone (strictly decreasing along the columns).

{\bf Step 1}  We verify the log-convexity condition \eqref{eq:logconv}. Fixing $n$, consider the ratio
\[
r_j^n=\frac{a_{n-j}^n}{a_{n-1-j}^{n-1}}=\frac{(t_n-t_{j-1})^{\alpha}
-(t_n-t_j)^{\alpha}}{(t_{n-1}-t_{j-1})^{\alpha}
-(t_{n-1}-t_{j})^{\alpha}}.
\]
We now verify that $r_j^n$ is decreasing for $j$.
To do this, we consider the function
\[
\theta(x, y)=\frac{(t_n-t_{j-1}-x)^{\alpha}
-(t_n-t_j-y)^{\alpha}}{(t_{n-1}-t_{j-1}-x)^{\alpha}
-(t_{n-1}-t_{j}-y)^{\alpha}}, \quad 0\le x\le \tau_j, ~0\le y\le \tau_{j+1}.
\]
Clearly,  $0\le \theta \le 1$.
Moreover, 
\[
\frac{\partial\theta}{\partial x}=[(t_{n-1}-t_{j-1}-x)^{\alpha}
-(t_{n-1}-t_{j}-y)^{\alpha}]^{-2}\alpha A_x,
\]
where
\begin{multline*}
A_x=(t_{n-1}-t_{j-1}-x)^{\alpha-1}[(t_n-t_{j-1}-x)^{\alpha}
-(t_n-t_j-y)^{\alpha}]\\
-(t_n-t_{j-1}-x)^{\alpha-1}[(t_{n-1}-t_{j-1}-x)^{\alpha}
-(t_{n-1}-t_{j}-y)^{\alpha}].
\end{multline*}
To show $A_x\le 0$, consider the mapping
\begin{multline*}
z\mapsto A_x(z):=
(t_{n-1}-t_{j-1}-x)^{\alpha-1}[(t_n-t_{j-1}-x)^{\alpha}
-(t_n-t_{j-1}-x-z)^{\alpha}] -\\
(t_n-t_{j-1}-x)^{\alpha-1}[(t_{n-1}-t_{j-1}-x)^{\alpha}
-(t_{n-1}-t_{j-1}-x-z)^{\alpha}], 0\le z\le \tau_j+y-x.
\end{multline*}
One can find that $A_x(0)=0$ and 
\[
A_x'(z)=\alpha (t_{n-1}-t_{j-1}-x)^{\alpha-1} (t_n-t_{j-1}-x-z)^{\alpha-1}
-\alpha (t_n-t_{j-1}-x)^{\alpha-1}(t_{n-1}-t_{j-1}-x-z)^{\alpha-1}.
\]
Since $(t_n-t_{j-1}-x)(t_{n-1}-t_{j-1}-x-z)>
(t_{n-1}-t_{j-1}-x) (t_n-t_{j-1}-x-z)$ for $z>0$, one has $A_x'(z)<0$.  Hence, $A_x=A_x(\tau_j+y-x)<0$.

For the derivative on $y$, the calculation is similar. In fact
\[
\frac{\partial\theta}{\partial y}=[(t_{n-1}-t_{j-1}-x)^{\alpha}
-(t_{n-1}-t_{j}-y)^{\alpha}]^{-2}\alpha A_y,
\]
with 
\begin{multline*}
A_y=(t_n-t_{j}-y)^{\alpha-1}[(t_{n-1}-t_{j-1}-x)^{\alpha}
-(t_{n-1}-t_{j}-y)^{\alpha}] \\
-(t_{n-1}-t_{j}-y)^{\alpha-1}[(t_n-t_{j-1}-x)^{\alpha}
-(t_n-t_j-y)^{\alpha}].
\end{multline*}
Here, we consider 
\begin{multline*}
z\mapsto A_y(z):=(t_n-t_{j}-y)^{\alpha-1}[(t_{n-1}-t_{j}-y+z)^{\alpha}
-(t_{n-1}-t_{j}-y)^{\alpha}] \\
-(t_{n-1}-t_{j}-y)^{\alpha-1}[(t_n-t_{j}-y+z)^{\alpha}-(t_n-t_j-y)^{\alpha}],
 \quad 0\le z\le y+\tau_j-x.
\end{multline*}
Using similar trick, one can show that $A_y(0)=0$ and $A_y'(z)< 0$ for $z>0$.
Hence, $A_y=A_y(y+\tau_j-x)<0$. Hence, $\theta$ is decreasing on the region considered. 
This then verifies that $r_j^n> r_{j+1}^n$, and thus \eqref{eq:logconv}.

{\bf Step 2}  We consider $\Gamma(1+\alpha)L^{(-1)}\pc A\pc L=(\beta_{n-j}^n)$. Then, we show the log-convexity condition \eqref{eq:logconv} for this kernel.
It is not hard to determine that
\[
\beta_{n-j}^n=(t_n-t_{j-1})^{\alpha}-(t_{n-1}-t_{j-1})^{\alpha}.
\]
It is straightforward to see that every element is positive. Note that this is row monotone, and may not be column monotone.

To verify \eqref{eq:logconv}, the ratio considered would be different
\[
\bar{r}_j^n=\frac{\beta_{n-j}^n}{\beta_{n-(j-1)}^n} =\frac{(t_n-t_{j-1})^{\alpha}-(t_{n-1}-t_{j-1})^{\alpha}}{(t_{n}-t_{j-2})^{\alpha}-(t_{n-1}-t_{j-2})^{\alpha}}.
\]
We will show that this is decreasing in $n$. The calculation is similar. Define
\[
\bar{\theta}(x, y) =\frac{(t_n-t_{j-1}+x)^{\alpha}-(t_{n-1}-t_{j-1}+y)^{\alpha}}{(t_{n}-t_{j-2}+x)^{\alpha}-(t_{n-1}-t_{j-2}+y)^{\alpha}}, \quad 0\le x\le \tau_{n+1}, 0\le y\le \tau_n.
\]
It can be shown similarly that
\[
\frac{\partial\bar{\theta}}{\partial x}<0, \quad \frac{\partial\bar{\theta}}{\partial y}<0.
\]
 This then implies that $\bar{\theta}$ is decreasing on the region considered.  Hence, we find
\[
\bar{r}_j^n> \bar{r}_j^{n+1}.
\]
This is equivalent to \eqref{eq:logconv}.

Hence,  $A$ and $L^{(-1)}\pc A \pc L$ verify the conditions in Proposition \ref{pro:suffcond}, so $A$ is R-CMM.
\end{proof}

\subsection{Numerical illustration}

In this subsection, we perform some numerical tests for the scheme mentioned in section \ref{sec:ex}. In particular, we will take
\begin{gather}
D_c^{\alpha}u=\sin(1+u^2),\quad u(0+)=u_0 \Longleftrightarrow u(t)=u_0+\frac{1}{\Gamma(\alpha)}\int_0^t
(t-s)^{\alpha-1}f(u(s))\,ds
\end{gather}
as the example, where $\alpha=0.6$.
We use three types of meshes:
\begin{itemize}
\item Increasing mesh with $\tau_1=0.01$ and $\tau_{j+1}/\tau_j=1.2$
for $j\ge 1$;
\item Decreasing mesh with $\tau_1=0.1$ and $\tau_j=0.1(1+0.5j)^{-1/2}$;
\item Random mesh with $\tau_j\sim 0.1\cdot U(0, 1)$, where $U(0, 1)$ indicates the uniform distribution on $(0, 1)$.
\end{itemize}

\begin{figure}[!htbp]
\centering  
\includegraphics[width=0.95\textwidth]{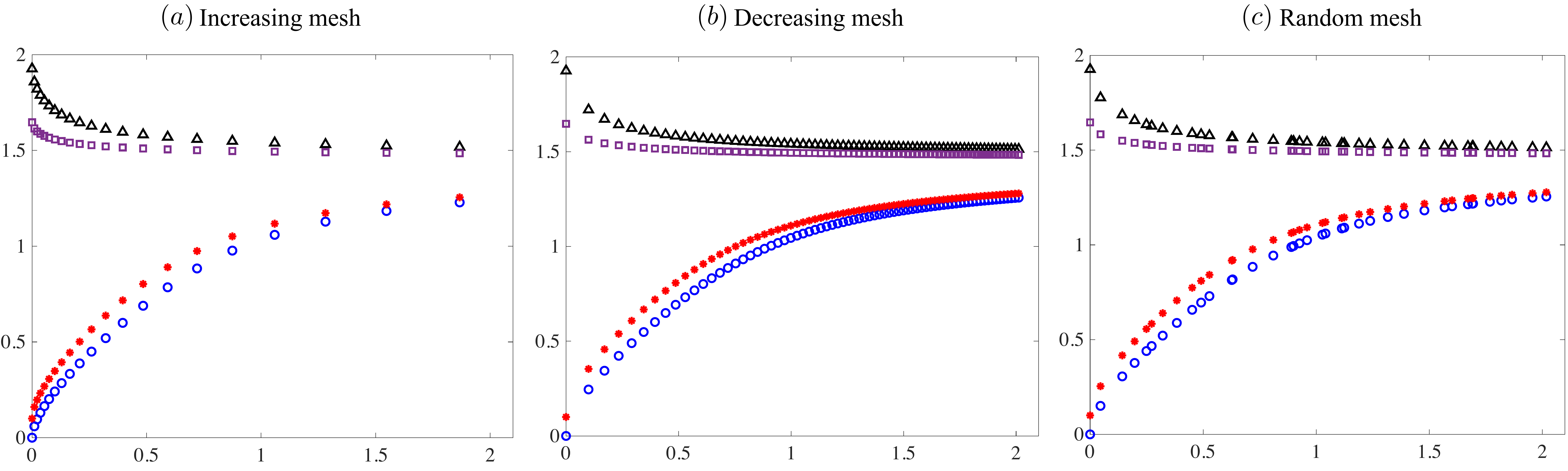}  
\caption{The numerical solutions with different meshes and different initial values. (a) the increasing mesh; (b) the decreasing mesh; (c) the random mesh. We find that the orders of the solution curves indeed stay monotone, and each curve for different meshes is also monotone.} 
 \label{fig:example1} 
\end{figure}

The numerical results with initial values $u_0=0, 0.1, \sqrt{3\pi/2-1}, \sqrt{3\pi/2-2}$ are shown in Fig. \ref{fig:example1} .
We find that the two monotonicity properties are preserved for these numerical solutions over each nonuniform mesh. 
(In the figure, the monotonicty for the $u_0=\sqrt{3\pi/2-2}$ curve is not very clear. With the concrete numerical values, we have checked that they are indeed monotone.) This then verifies the theory.

\section*{Acknowledgement}

This work was financially supported by the National Key R\&D Program of China, Project Number 2021YFA1002800 and 2020YFA0712000. The work of Y. Feng was partially sponsored by NSFC 12301283, Shanghai Sailing program 23YF1410300 and Science and Technology Commission of Shanghai Municipality (No. 22DZ2229014). The work of L. Li was partially supported by NSFC 12371400 and 12031013,  Shanghai Science and Technology Commission (Grant No. 21JC1403700, 20JC144100), the Strategic Priority Research Program of Chinese Academy of Sciences, Grant No. XDA25010403.  L. Li would like to thank Jiwei Zhang for his comment on CM-preserving schemes which brought our attention to nonuniform meshes.

\appendix

\section{The proofs for the monotonicity properties}\label{app:monoproof}

We first present the proof for Theorem \ref{thm:contcomparison}.
\begin{proof}[Proof of Theorem \ref{thm:contcomparison}]
Let $v(t):=u_1(t)-u_2(t)$ and $[0, T_b)$ be the common interval of existence. Then, 
\[
v(t)=\gamma(t)+\int_0^t a(t-s)g(s) v(s)\,ds, \quad 
g(s):=\int_0^1 \partial_u f(s, z u_1(s)+(1-z) u_2(s))\,dz.
\]
Convolving both sides with $\delta-r_{\lambda}$, one has
\[
v(t)-r_{\lambda}*v(t)=(\delta-r_{\lambda})*\gamma(t)+\lambda^{-1}r_{\lambda}*(gv).
\]
This implies that
\begin{gather}\label{eq:diffresol}
v(t)=\beta_{\lambda}(t)+\int_0^t r_{\lambda}(t-s)[1+\lambda^{-1}g(s)]v(s)\,ds.
\end{gather}
For any $t\in [0, T_b)$, $u_1, u_2$ are bounded on $[0, t]$ and hence one can choose $\lambda$ large enough to make $1+g(s)/\lambda>0$ on $[0, t]$.  Recall $\beta_{\lambda}(t)\ge 0$. Since $\beta_{\lambda}\equiv 0$ is trivial, we assume that $\beta_{\lambda}>0$ somewhere. 

If $\beta_{\lambda}(0)>0$ (or equivalently $\gamma(0)>0$), then $v(0)=\beta_{\lambda}(0)>0$ so that $v(s)> 0$ for $s$ small enough. If $v$ ever reaches $0$ at a first time $t_*$, then \eqref{eq:diffresol} tells us that
\[
0=v(t_*)=\beta_{\lambda}(t_*)+\int_0^{t_*} r_{\lambda}(t_*-s)[1+\lambda^{-1}g(s)]v(s)\,ds>0.
\] 
This is a contradiction, so $v(t)>0$ for all $t\in [0, T_b)$. 

Now, consider the degenerate case, namely $\beta_{\lambda}(t)$ is zero on $[0, t_1]$ and  $\beta_{\lambda}(t)>0$ on $(t_1, t_1+\delta)$ for some $\delta>0$. 
Then, by \eqref{eq:interres}, $\gamma=\beta_{\lambda}+\lambda a*\beta_{\lambda}$ is also zero on $[0, t_1]$. By the uniqueness of the solution to \eqref{eq:vol}, $u_1(t)=u_2(t)$ or $v(t)=0$ on $[0, t_1]$.
Assume for the purpose of contradiction that $v(s)<0$
on $(t_1, t_1+\delta_1)$ for some $\delta_1\le \delta$.
Fix $\lambda$ such that $1+h(s)/\lambda>0$ on $(t_1, t_1+\delta_1)$.
Let $A:=\sup_{s\in (t_1, t_1+\delta_1)}(1+h(s)/\lambda)>0$.
We take $\e\in (0, \delta_1]$ such that $\int_{0}^{\e}r_{\lambda}(s)\,ds<1/(2A)$. Let $t_2=\mathrm{argmin}_{s\in [t_1, t_1+\e]}v(s)\in (t_1, t_1+\e]$. Then,
\begin{multline*}
v(t_2)=\beta_{\lambda}(t_2)+\int_{t_1}^{t_2}r_{\lambda}(t_2-s)(1+h(s)/\lambda)v(s)\,ds \\
\ge \beta_{\lambda}(t_2)+Av(t_2)\int_{t_1}^{t_2}r_{\lambda}(t_2-s)\,ds \ge \beta_{\lambda}(t_2)+v(t_2)/2.
\end{multline*}
This is a contradiction. Therefore, $v(t)\ge 0$ for $t\in (t_1, t_1+\delta_1)$ for some $\delta_1\le \delta$, which can be strengthened to $v(t)>0$ for $t\in (t_1, t_1+\delta_1)$ by \eqref{eq:diffresol}. Then, \eqref{eq:diffresol} again implies that $v(t)>0$ for $t\in (t_1, T_b)$.
\end{proof}

Below we present the proof for Theorem \ref{thm:mon4auto}. 

\begin{proof}[Proof of Theorem \ref{thm:mon4auto} ]
Convolving both sides of the equation with $\delta-r_{\lambda}$, one has
\[
u(t)=h(t)+(u-h)*r_{\lambda}+\lambda^{-1}r_{\lambda}*f(\cdot, u(\cdot)).
\]
Taking derivative on both sides which is feasible since $u$ is absolutely continuous, one has
\begin{multline}\label{eq:uprime}
u'(t)=h'(t)+r_{\lambda}(t)(u(0)-h(0))+r_{\lambda}*(u'-h')\\
+\lambda^{-1}r_{\lambda}(t)f(0, u(0))
+\lambda^{-1}r_{\lambda}*(\partial_t f+\partial_u f u').
\end{multline}
Sending $t\to 0^+$ in the equation, one has $u(0)=h(0)$. Then, the equation is reduced to
\[
u'(t)=g(t)+r_{\lambda}*([1+\lambda^{-1}\partial_u f] u'),
\]
where
\[
g(t)=(\delta-r_{\lambda})*h'+\lambda^{-1}r_{\lambda}(t)f(0, h(0))+\lambda^{-1}r_{\lambda}*\partial_t f(t, u).
\]
The conclusions hold by similar arguments as in the proof of Theorem \ref{thm:contcomparison}.
\end{proof}

\section{Sketch of the proof for the time continuous CMM property}\label{app:contcmm}

\begin{proof}
Since $a$ is CMM if and only if $a^c$ is CMM, hence if we can show the equivalence between (a) and (b), then the equivalence between (a) and (c) follows automatically. 

(a) $\Rightarrow$ (b):  let $a^c=\tilde{\alpha}^c\delta+k$ be the complementary kernel. Then, $k$ is nonnegative and nonincreasing and integrable on $(0, T)$.  Consider the kernel $a_{\e}$ defined by
\[
a_{\e}*((\epsilon+\tilde{\alpha}^c)\delta+k)=1_{t\ge 0}.
\]
Then, $a_{\e}$ is absolutely continuous. Taking $t\to 0^+$, one deduces that
\[
\e+\tilde{\alpha}^c=1/a_{\e}(0).
\]
Let $r_{\e,\lambda}$ be the resolvent for $a_{\e}$. Using the intuition that $\delta- r_{\e,\lambda}=\lambda^{-1}a_{\e}^{(-1)}*r_{\e,\lambda}$, one may obtain that
\[
1-1*r_{\e,\lambda}=\lambda^{-1}r_{\e,\lambda}*a_{\e}^c=(\lambda a_{\e}(0))^{-1}r_{\e,\lambda}
+\lambda^{-1}r_{\e,\lambda}*k.
\]
This can actually be justified rigorously and see the proof in \cite[Theorem 2.2]{clement1981asymptotic}.
With this, one has an equation for $r_{\e,\lambda}$:
\[
r_{\e,\lambda}+a_{\e}(0) ( k+\lambda)*r_{\e,\lambda}=\lambda a_{\e}(0).
\]
Since $(k+\lambda)$ is positive and nonincreasing, one can then show that $r_{\e,\lambda}>0$.
Then, $1-1*r_{\e,\lambda}=(\lambda a_{\e}(0))^{-1}r_{\e,\lambda}
+\lambda^{-1}r_{\e,\lambda}*k \ge 0$ follows.

Next, one aims to take the limit $\e\to 0^+$. This limit is not straightforward as the atom may appear. The approach is to convolving with absolutely continuous functions $z$ with $z(0)=0$ and $z\ge 0$. Then, show $u_{\e}:=r_{\e,\lambda}*z\to r_{\lambda}*z=:u$ by considering their equations. 
In particular, $r_{\lambda}*(\lambda^{-1}a^c+1_{t\ge 0})=1_{t\ge 0}$, one has
\[
(\lambda 1_{t\ge 0}+\tilde{\alpha}^c\delta+k)*u=\lambda 1_{t\ge 0}* z, 
\quad \e u_{\e}+(\lambda 1_{t\ge 0}+\tilde{\alpha}^c\delta+k)*u_{\e}=\lambda 1_{t\ge 0}* z
\]
Taking the difference,
\[
\e (u_{\e}-u)+(\lambda 1_{t\ge 0}+\tilde{\alpha}^c\delta+k)*(u_{\e}-u)=-\e u.
\]
One readily shows that $\int_0^t |u_{\e}-u|\,ds\to 0$. This implies that $u\ge 0$.
Then, $r_{\lambda}\ge 0$.  The sign of $s$ is similarly proved using the relation between $r_{\lambda}$
and $s_{\lambda}$.

(b) $\Rightarrow$ (a): the main idea is that $\lambda s_{\lambda}*(\lambda^{-1}\delta+a)=1_{t\ge 0}$.
It is expected that $s_{\lambda}\to 0$ as $\lambda\to \infty$ since $r_{\lambda}$ is close to $\delta$
as $\lambda\to\infty$. Hence, the goal is to take certain limit of $\lambda s_{\lambda}$ such that
the limit would be the complementary kernel, which is then nonnegative and nonincreasing.

By the conditions given, it can be shown that $\lambda s_{\lambda}$ is uniformly bounded  by $(\int_0^t b(\tau)\,d\tau)^{-1}$ which is uniformly bounded on $[\e, T]$ for any $\e>0$.
To find a convergent subsequence, one regard $\lambda s_{\lambda}$ as a family of measures and then consider the topology tested against the absolutely continuous functions. The narrow limit is then the complementary kernel, which could possibly have an atom at $t=0$. This intuition can be made rigorous and see the proof of \cite[Theorem 2.2]{clement1981asymptotic}.

\end{proof}

\bibliographystyle{plain}
\bibliography{frac}

\end{document}